\documentclass[11pt]{article}

\usepackage[T1]{fontenc}
\usepackage{lmodern}
\usepackage{microtype}
\usepackage[a4paper,margin=27mm]{geometry}
\usepackage{amsmath,amssymb,amsthm,mathtools}
\usepackage{xcolor}
\usepackage[colorlinks=true,linkcolor=blue!55!black,
  citecolor=blue!55!black,urlcolor=blue!55!black]{hyperref}

\newtheorem{theorem}{Theorem}[section]
\newtheorem{lemma}[theorem]{Lemma}
\newtheorem{problem}[theorem]{Problem}
\newtheorem{proposition}[theorem]{Proposition}
\newtheorem{corollary}[theorem]{Corollary}
\theoremstyle{definition}

\newcommand{\cc}{\operatorname{cc}}
\newcommand{\cpn}{\operatorname{cp}}

\newcommand{\calQ}{\mathcal Q}
\newcommand{\join}{\vee}

\title{On the difference between clique partition
and clique covering numbers}
\author{Bo Ning\thanks{College of Computer Science, Nankai University, Tianjin 300350, P.R. China. E-mail: \texttt{bo.ning@nankai.edu.cn}. Partially supported by the National Natural Science Foundation of China (Nos. 12371350 and 12426675) and the Fundamental Research Funds for the Central Universities, Nankai University (No. 63243151).}}
\date{30 July 2026}

\begin{document}
\maketitle

\begin{abstract}
For a graph \(G\), let \(\operatorname{cp}(G)\) and
\(\operatorname{cc}(G)\) be the minimum numbers of cliques in an edge
partition and a clique cover of \(G\), respectively. Set
$\sigma_n
 = \max_{\lvert V(G)\rvert=n}
   \bigl(\operatorname{cp}(G)-\operatorname{cc}(G)\bigr),
d_n=\left\lfloor\frac{n^2}{4}\right\rfloor-\sigma_n.$
In 1983, Erd\H{o}s, Faudree, and Ordman asked whether \(d_n=O(n)\). Caccetta,
Erd\H{o}s, Ordman, and Pullman previously constructed graphs showing
\(d_n=O(n^{3/2})\). We prove that
$d_n=\Theta(n^{4/3}),$
thereby determining the correct order of the deficit and answering
their question in the negative.
\end{abstract}

\section{Introduction}
All graphs are
finite, simple, and undirected.
A \emph{clique covering} of a graph $G$ is a family of cliques
whose union contains $E(G)$.  If every edge lies in exactly one member
of the family, then the family is a \emph{clique partition}.  We denote the
minimum possible sizes of a clique
covering and a clique partition by $\cc(G)$ and $\cpn(G)$.
These notions originate from set representations of graphs; see Erd\H{o}s, Goodman, and P\'osa \cite{EGP1966}. A clique covering gives each vertex the set of covering cliques containing it, with adjacency equivalent to intersecting sets; the correspondence is bijective. Requiring any two representing sets to meet in at most one element turns every edge into a unique clique, yielding a clique partition. In incidence-geometric language, such representations are partial linear spaces, and for complete graphs, finite linear spaces (see de Bruijn–Erd\H{o}s \cite{deBruijnErdos1948}, 1948).

Erd\H{o}s, Goodman, and P\'osa \cite{EGP1966}  proved that every
$n$-vertex graph $G$ satisfies
\begin{equation}\label{eq:egp}
  \cpn(G)\leq\left\lfloor\frac{n^2}{4}\right\rfloor.
\end{equation}
The balanced complete bipartite graph shows that
\eqref{eq:egp} is best possible.  For this extremal graph, however,
every clique has order two, and hence $\cc(G)=\cpn(G)$.
Every clique partition is
a clique cover, so
$\cc(G)\leq \cpn(G)$.
Orlin \cite{Orlin1977} proved that determining $\cc(G)$ for a graph is NP-complete.

Define $$\sigma(G):=\cpn(G)-\cc(G)$$ and 
$$\sigma_n:=\max_{|V(G)|=n}\bigl(\cpn(G)-\cc(G)\bigr).$$
Caccetta, Erd\H{o}s, Ordman, and
Pullman \cite{CEOP1985} called $\sigma(G)$ the spread of $G$.
They proved that $\frac{n^2}{4}-\frac{n^{3/2}}{2}+\frac n4\leq \sigma_n\leq \lfloor n^2/4\rfloor-2$, and
constructed a sequence of $n$-vertex graphs $G_n$ satisfying
\begin{equation}\label{eq:old-additive}
 \sigma(G_n)
 =\frac{n^2}{4}-\frac{n^{3/2}}{2}+\frac n4+O(1).
\end{equation}
 Consequently
$\sigma_n\sim n^2/4$, while their construction gives
$d_n=O(n^{3/2})$.  What remained open was the order of the deficit $d_n=\left\lfloor\frac{n^2}{4}\right\rfloor-\sigma_n.$

According to Caccetta, Erd\H{o}s, Ordman, and Pullman~\cite{CEOP1985}, Erd\H{o}s asked in
1982 for the maximum possible spread. The later \(O(n)\)-deficit
question was recorded by Erd\H{o}s Problems page
\cite{ErdosProblems632}, and was contributed to
Erd\H{o}s, Faudree, and Ordman in 1983.

\begin{problem}[Erd\H{o}s--Faudree--Ordman]\label{prob:erdos}
Is there a sequence of graphs $G_n$, with $|V(G_n)|=n$, such that
\[
 \cpn(G_n)-\cc(G_n)=\frac{n^2}{4}+O(n)?
\]
\end{problem}

The above problem also
appears as Problem 66 in Chung \cite{Chung1997} and as Question 6.32 in Faudree, Rousseau, and Schelp
\cite[Questions~6.31--6.32]{FRS1997}, adjacent to the corresponding ratio problem about the largest ratio
$\cpn(G)/\cc(G)$.

Our main result determines the correct
order of $\frac{n^2}{4}-\sigma_n$. 

\begin{theorem}\label{thm:main}
There exist absolute constants $C_1,C_2>0$ such that
\[
 \left\lfloor\frac{n^2}{4}\right\rfloor-C_1n^{4/3}
 \leq \sigma_n\leq
 \left\lfloor\frac{n^2}{4}\right\rfloor-C_2n^{4/3}
\]
for every sufficiently large integer $n$.
\end{theorem}

Equivalently, Theorem~\ref{thm:main} states that
$d_n=\Theta(n^{4/3})$.  It therefore answers
Problem~\ref{prob:erdos} in the negative and identifies the
second-order term left open by \eqref{eq:old-additive}. 

Erd\H{o}s, Faudree, and Ordman
\cite[Lemma~4]{EFO1988} proved
a graph $G$ with $s$ crossing edges, $a$
edges inside one side, and $b$ edges inside the other side satisfies that
\[
  \cpn(G)\geq s-a-b-\min\{a,b\}.
\]
Orlin \cite{Orlin1977} also studied $cc(G)$ when $G$ is a line graph. Wallis \cite{Wallis1982} showed that if $G$ has $o(\sqrt{n})$ vertices  then $\cpn(K_n-G)$ is asymptomatically equal to $n$. Related problems were studied by Erd\H{o}s, Faudree, and Ordman
\cite{EFO1988}, who proved that, if
$m=cn^a$ with $1/2<a<1$, then
$\cpn(K_n-K_m)\sim c^2n^{2a}=m^2$; they also constructed graphs satisfying
$\cpn(G)/\cc(G)>n^2/64$ and showed that the clique covering number of
$K_n$ with a matching removed lies between $\log n-1$ and
$2\log n$. For some related work and variants, see Erd\H{o}s,
Ordman, and Zalcstein \cite{EOZ1987,EOZ1993}. For a survey on graph covering and partitioning problems, see Schwartz \cite{Schwartz2022}. 

\medskip

\noindent
{\bf Overview.}
For the construction, we determine both clique parameters of
\(G_{h,k}=(hK_k)\vee(hK_k)\); choosing \(k\asymp n^{1/3}\) and
\(h\asymp n^{2/3}\) gives \(d_n=O(n^{4/3})\). Conversely, a
minimum-weight clique partition produces a dense triangle-free spanning
subgraph and a nearly balanced cut. The covering number gives a lower
bound on the number of internal edges, while triangle repacking gives
an upper bound; comparison yields \(d_n=\Omega(n^{4/3})\).

\medskip

\noindent
{\bf Landau's notation.} We use Landau's notation. Throughout the paper, all asymptotic notation is taken as
$n\to\infty$.  If $f_1=f_1(n)$ and $f_2=f_2(n)>0$, then
$f_1=O(f_2)\quad\Longleftrightarrow\quad
 |f_1(n)|\le Cf_2(n)$
for some constant $C>0$ and all sufficiently large $n$, while
$f_1=\Omega(f_2)\quad\Longleftrightarrow\quad
 f_1(n)\ge cf_2(n)$
for some constant $c>0$ and all sufficiently large $n$.
We write $f_1\lesssim f_2$ and $f_1\gtrsim f_2$ as synonyms for
$f_1=O(f_2)$ and $f_1=\Omega(f_2)$, respectively.  Finally,
$f_1\asymp f_2$ (that is, $f_1=\Theta(f_2)$)
means that there exist constants $c,C>0$ such that
$cf_2(n)\le f_1(n)\le Cf_2(n)$ for all sufficiently large $n$.
Unless explicitly indicated otherwise, all constants implicit in this
notation are absolute.

\medskip

\noindent
{\bf Organization.}
In Section~\ref{sec:preliminaries}, we introduce notation and
state the results used later.  In Section~\ref{sec:construction}, we give the lower bound construction and compute its two clique parameters.
In Section~\ref{sec:upper}, we  prove the upper bound and complete the proof
of Theorem~\ref{thm:main}.

\section{Preliminaries}\label{sec:preliminaries}
A clique in a
clique cover or clique partition is allowed to have order two.  For a graph
$G$, let $e(G)=|E(G)|$, $\alpha(G)$ denote its independence
number, and $N_G(v)$ be the open neighborhood of $v$.
For a clique partition $\calQ$ of $G$, define
$w(\calQ)=\sum_{Q\in\calQ}|Q|.$

We use four results. The following result was conjectured by Katona and Tarj\'an and
proved independently by Gy\H{o}ri and Kostochka \cite{GK1979},
Chung \cite{Chung1981}, and Kahn \cite{Kahn1981}. 
\begin{theorem}[Gy\H{o}ri--Kostochka \cite{GK1979}, Chung
  \cite{Chung1981}, and Kahn \cite{Kahn1981}]
\label{thm:weighted}
Every graph $G$ on $n$ vertices has a clique partition $\calQ$
satisfying
\[
  w(\calQ)\leq2\left\lfloor\frac{n^2}{4}\right\rfloor.
\]
\end{theorem}

\begin{theorem}[Caro \cite{Caro1979} and Wei \cite{Wei1981}]
\label{thm:carowei}
If $G$ has $n$ vertices and $m$ edges, then
\[
 \alpha(G)\geq
 \sum_{v\in V(G)}\frac{1}{d_G(v)+1}
 \geq\frac{n^2}{2m+n}.
\]
Consequently, if $\alpha(G)\leq x$, then
\begin{equation}\label{eq:carowei-consequence}
  e(G)\geq\frac{n^2}{2x}-\frac{n}{2}.
\end{equation}
\end{theorem}

\begin{theorem}[Vizing \cite{Vizing1964}]\label{thm:vizing}
Every graph $G$ has a proper edge-coloring with at most
$\Delta(G)+1$ colors.
\end{theorem}

\begin{lemma}[{Erd\H{o}s--Faudree--Ordman
  \cite[Lemma~4]{EFO1988}}]\label{lem:efo}
Let $G$ be a graph and let $V(G)=A\mathbin{\dot\cup}B$. If $G$ has $a$ edges in $A$, $b$ edges in $B$ and $s$ edges connecting $A$ and $B$, then
\[
  \cpn(G)\geq s-a-b-\min\{a,b\}.
\]
\end{lemma}

\section{The proofs}
\subsection{The lower bound}\label{sec:construction}

The lower-bound graph is obtained by joining two cluster graphs.  For
positive integers $h$ and $k$, let $hK_k$ denote the disjoint
union of $h$ copies of $K_k$, and define
$G_{h,k}=(hK_k)\join(hK_k).$
We write
\[
 A=A_1\dot\cup\cdots\dot\cup A_h,
 \qquad
 B=B_1\dot\cup\cdots\dot\cup B_h,
\]
where each $A_i$ and each $B_j$ induces a copy of $K_k$.  There are no
edges between distinct cliques on the same side, and all edges between
$A$ and $B$ are present.

For even $k$, both clique parameters can be determined exactly.

\begin{proposition}\label{prop:exact}
If $k$ is even and $h\geq k-1$, then
$\cc(G_{h,k})=h^2$ and
$\cpn(G_{h,k})
  =h^2k^2-\frac32hk(k-1).$
\end{proposition}

\begin{proof}
The vertex sets
$A_i\cup B_j,$  $1\leq i,j\leq h,$ induce cliques, i.e. $G[A_i]\cong G[B_j]=K_k$ and
form a clique cover of size $h^2$, so $\cc(G_{h,k})\leq h^2$.  For the
reverse inequality, choose $a_i\in A_i$ and $b_j\in B_j$ for all
$i,j$.  These
vertices induce a copy of $K_{h,h}$.  Every clique of $G_{h,k}$ contains at
most one of the selected vertices in $A$ and at most one of the
selected vertices in $B$, and consequently covers at most one edge
of this $K_{h,h}$.  Every clique covering therefore has at least
$h^2$ members, and hence
$\cc(G_{h,k})=h^2.$

We next determine the partition number. Apply Lemma~\ref{lem:efo} to
the cut $(A,B)$.  The numbers of crossing and internal edges are
$s=h^2k^2, u=v=h\binom k2.$
It follows that
\begin{equation}\label{eq:construction-lower}
 \cpn(G_{h,k})
 \geq h^2k^2-3h\binom k2
 =h^2k^2-\frac32hk(k-1).
\end{equation}

It remains to construct a partition
attaining equality in \eqref{eq:construction-lower}.  Since $k$ is
even, $K_k$ has a one-factorization
\cite[Section~3.3]{Stinson2004}.  Label its $k-1$
one-factors by $0,\ldots,k-2$, using the same labeling in each 
$A_i$ and each $B_j$. All cluster
indices are interpreted modulo $h$.

For every $i\in\{0,\ldots,h-1\}$ and
$r\in\{0,\ldots,k-2\}$, pair the $k/2$ edges in factor $r$
of $A_i$ bijectively with the $k/2$ edges in factor $r$ of
$B_{i+r}$, and include in the partition the copy of $K_4$
induced by the four endpoints of each
pair of matched edges in the partition.  Since $h\geq k-1$, the cluster pairs
$(A_i,B_{i+r})$ are distinct as $r$ varies.

Every edge internal to $A$ now lies in exactly one of these copies of
$K_4$.  The same holds in $B$: an edge in factor $r$ of
$B_j$ is paired with an edge of $A_{j-r}$.  No crossing edge is
repeated.  Within a fixed cluster pair, the factor edges form
matchings, so the corresponding $K_4$'s are vertex-disjoint, while
distinct index pairs $(i,r)$ correspond to distinct ordered cluster pairs.

There are
\[
  h\binom k2=\frac12hk(k-1)
\]
such copies of $K_4$, and they contain
\[
  4h\binom k2=2hk(k-1)
\]
crossing edges.  Use every remaining crossing edge as a two-vertex
clique.  The resulting clique partition has
\begin{align*}
 h\binom k2+
 \left(h^2k^2-4h\binom k2\right)
 &=h^2k^2-3h\binom k2=h^2k^2-\frac32hk(k-1)
\end{align*}
members.  Together with \eqref{eq:construction-lower}, this proves the assertion.
\end{proof}

By balancing the two error terms in Proposition~\ref{prop:exact}, we give the
lower bound in Theorem~\ref{thm:main}.

\begin{corollary}\label{cor:lower}
There is an absolute constant $C>0$ such that
\[
  \sigma_n\geq
  \left\lfloor\frac{n^2}{4}\right\rfloor-Cn^{4/3}
\]
for every sufficiently large integer $n$.
\end{corollary}

\begin{proof}
Let $k$ be the smallest even integer at least $n^{1/3}$, and
put
$h=\left\lfloor\frac{n}{2k}\right\rfloor.$
Then
$n^{1/3}\leq k<n^{1/3}+2,$
and $h\geq k-1$ for all sufficiently large $n$.  Add
$n-2hk$ isolated vertices to $G_{h,k}$; this changes neither
$\cpn$ nor $\cc$.  Proposition~\ref{prop:exact} gives
\begin{align*}
 \cpn(G_{h,k})-\cc(G_{h,k})=h^2k^2-\frac32hk(k-1)-h^2=\frac{(2hk)^2}{4}-O(hk^2+h^2).
\end{align*}
\noindent Since $hk=O(n)$, $k=O(n^{1/3})$, and
$h=O(n^{2/3})$, we have
$hk^2+h^2=O(n^{4/3}).$
Furthermore, $0\leq n-2hk<2k$, and hence
\[
  0\leq\frac{n^2-(2hk)^2}{4}
  =\frac{(n-2hk)(n+2hk)}4
  =O(nk)=O(n^{4/3}).
\]
Increasing $C$ accounts for
the rounding in $h=\lfloor\frac{n}{2k}\rfloor$.
\end{proof}

\subsection{The upper bound}\label{sec:upper}

We first prove the upper bound for graphs with even order.
Throughout this
section, let $G$ be a graph on $2n$ vertices, and write
$p=\cpn(G),t=\cc(G),\delta=n^2-p,D=\delta+t.$
\noindent By \eqref{eq:egp}, both $\delta$ and $D$ are
nonnegative. 

\begin{proposition}\label{prop:upper-even}
There is an absolute constant $c_0>0$ such that every graph $G$ on
$2n$ vertices, with $n$ sufficiently large, satisfies
\[
  \cpn(G)-\cc(G)\leq n^2-c_0n^{4/3}.
\]
\end{proposition}

The first lemma extracts a dense triangle-free skeleton from a
minimum-weight clique partition.

\begin{lemma}\label{lem:skeleton}
The graph $G$ has a triangle-free spanning subgraph $H$ satisfying
\[
  e(H)\geq n^2-3\delta.
\]
\end{lemma}

\begin{proof}
Choose a minimum-weight clique
partition $\calQ$;
it has no one-vertex member.  By
Theorem~\ref{thm:weighted},
$w(\calQ)\leq2n^2.$
Let $m=|\calQ|$, and let $r$ be the number of members of $\calQ$
of order at least three.  Since $m\geq p=n^2-\delta$, we have
\[
  2n^2\geq w(\calQ)\geq2m+r
  \geq2(n^2-\delta)+r.
\]
Thus
\begin{equation}\label{eq:number-large-blocks}
  r\leq2\delta.
\end{equation}

Let $H$ consist of the edges occurring as two-vertex members of
$\calQ$.  Since the members of $\calQ$ partition $E(G)$,
$e(H)=m-r\geq n^2-\delta-2\delta=n^2-3\delta.$
If $H$ contained a triangle, the three corresponding two-vertex
members of $\calQ$ could be
replaced by that triangle.  The total weight of these three
members would decrease
$6$ to $3$, contrary to the choice of $\calQ$.  Thus, $H$ is
triangle-free.
\end{proof}

We shall use the following elementary quantitative form of stability of Mantel's theorem.

\begin{lemma}\label{lem:cut}
Let $H$ be a triangle-free graph on $2n$ vertices, and write
\[
  e(H)=n^2-\eta.
\]
Then $V(H)$ has a partition $A\dot\cup B$ such that
$H[A]$ is empty, $e(H[B])\leq\eta$,
$e_{\overline H}(A,B)\leq2\eta,$
$\bigl||A|-n\bigr|\leq\sqrt{2\eta}.$
\end{lemma}

\begin{proof}
Mantel's theorem \cite{Mantel1907} gives $\eta\geq0$.  For
$v\in V(H)$, put
$B_v=V(H)\setminus N_H(v).$
For each $xy\in E(H)$, triangle-freeness implies
$N_H(x)\cap N_H(y)=\varnothing$.  Therefore,
\begin{align*}
 \sum_{v\in V(H)}e(H[B_v])
 &=\sum_{xy\in E(H)}
   \bigl(2n-d_H(x)-d_H(y)\bigr)=2ne(H)-\sum_{v\in V(H)}d_H(v)^2.
\end{align*}
By Cauchy--Schwarz inequality,
\[
  \sum_{v\in V(H)}d_H(v)^2
  \geq\frac{(2e(H))^2}{2n}
  =\frac{2e(H)^2}{n}.
\]
It follows that
\[
 \sum_{v\in V(H)}e(H[B_v])
 \leq\frac{2e(H)(n^2-e(H))}{n}
 \leq2n\eta.
\]
Choose $v$ such that $e(H[B_v])\leq\eta$, and set
$A=N_H(v), B=B_v.$
Then $H[A]$ is edgeless.  Write $a=|A|$,
$b=e(H[B])$, and $$\mu_H=e_{\overline H}(A,B)=|A||B|-e_H(A,B).$$  Since every
edge of $H$ either crosses the cut or lies in $B$,
\begin{align*}
 \mu_H
 &=a(2n-a)-\bigl(e(H)-b\bigr)=\eta+b-(a-n)^2.
\end{align*}
As $0\leq b\leq\eta$ and $\mu_H\geq0$, this identity yields
$(a-n)^2\leq\eta+b\leq2\eta$
and
$\mu_H\leq\eta+b\leq2\eta.$
This proves the lemma.
\end{proof}

We now transfer this cut to \(G\), bounding its missing crossing
edges and the clique partition numbers of \(G[A]\) and \(G[B]\).

\begin{lemma}\label{lem:transfer}
There is a partition $V(G)=A\dot\cup B$ such that, with
$\mu=e_{\overline G}(A,B)$,
\[
  \mu\leq6\delta,\qquad
  \bigl||A|-n\bigr|\leq\sqrt{6\delta},
\]
and each of $G[A]$ and $G[B]$ has a clique partition with at most
$5\delta$ members.
\end{lemma}

\begin{proof}
Apply Lemma~\ref{lem:cut} to the graph $H$ from
Lemma~\ref{lem:skeleton}, and put
\[
  \eta:=n^2-e(H)\leq3\delta.
\]
Since $H\subseteq G$, every crossing edge missing from $G$ is also
missing from $H$. Hence,
\[
  \mu\leq e_{\overline H}(A,B)
  \leq2\eta\leq6\delta,
\]
and Lemma~\ref{lem:cut} also gives
\[
  \bigl||A|-n\bigr|\leq\sqrt{2\eta}
  \leq\sqrt{6\delta}.
\]

Let $\calQ$ be the minimum-weight partition used in the proof of
Lemma~\ref{lem:skeleton}.  Intersect every member of $\calQ$ with
$A$, discarding intersections of order at most one.  This produces a
clique partition of $G[A]$.  Since $H[A]$ is empty, none of the
two-vertex members of $\calQ$ is retained; by
\eqref{eq:number-large-blocks}, the resulting partition has at most
$r\leq2\delta$ members.

The same restriction to $B$ may retain the $r$ members of
$\calQ$ of order at least three and the two-vertex members corresponding to $E(H[B])$.
Consequently $G[B]$ has a clique partition with at most
\[
  r+e(H[B])\leq2\delta+\eta\leq5\delta
\]
members.  The lemma follows.
\end{proof}

We next use the covering number to force many edges within the two
parts and the cut.

\begin{lemma}\label{lem:forcing}
For the partition in Lemma~\ref{lem:transfer}, put
$x=\alpha(G[A]),$ $y=\alpha(G[B]),$
 $L=e(G[A])+e(G[B]).$
Then
$xy\leq t+\mu\leq7D,$
and, if $D\leq n^2/12$, then
\[
  L\geq\frac{n^2-6D}{\sqrt{7D}}-n
  \geq\frac{n^2}{2\sqrt{7D}}-n.
\]
\end{lemma}

\begin{proof}
Choose independent sets $I\subseteq A$ and $J\subseteq B$ with
$|I|=x$ and $|J|=y$.  Every clique of $G$ contains at most one
vertex of $I$ and at most one vertex of $J$, and hence covers at
most one edge between $I$ and $J$.  At least $xy-\mu$ of these
edges are present.  A minimum clique covering therefore satisfies
$t\geq xy-\mu.$
Using Lemma~\ref{lem:transfer}, we obtain
\begin{equation}\label{eq:alpha-product}
  xy\leq t+\mu\leq t+6\delta\leq7D.
\end{equation}

Write $a=|A|$, so that $|B|=2n-a$.  Under the hypothesis
$D\leq n^2/12$, both sides of the cut are nonempty. Note that $D>0$. Indeed, as $x,y\geq 1$, the fact $xy\leq 7D$ implies that $D\geq 1/7$. Applying
\eqref{eq:carowei-consequence} to $G[A]$ and $G[B]$, and then
using the arithmetic--geometric mean inequality, gives
\begin{align*}
 L
 &\geq\frac{a^2}{2x}+\frac{(2n-a)^2}{2y}-n\geq\frac{a(2n-a)}{\sqrt{xy}}-n\geq\frac{n^2-(a-n)^2}{\sqrt{7D}}-n.
\end{align*}
By Lemma~\ref{lem:transfer},
\[
  (a-n)^2\leq6\delta\leq6D.
\]
This proves the first lower bound for $L$; the second follows from
$n^2-6D\geq n^2/2$.
\end{proof}

The reverse estimate comes from repacking internal edges into
edge-disjoint triangles.

\begin{lemma}\label{lem:repacking}
Let $A\dot\cup B$ be the partition in
Lemma~\ref{lem:transfer}, and choose $X\in\{A,B\}$ so that
$a_X=e(G[X])
  =\max\{e(G[A]),e(G[B])\}.$
If $D\leq n^2/10^4$, then
$a_X\leq36\delta.$
\end{lemma}

\begin{proof}
Put $Y=V(G)\setminus X$, $x_1=|X|$, and $y_1=|Y|$.  Define 
\[
  \mu=e_{\overline G}(X,Y)\leq6\delta.
\]
By Theorem~\ref{thm:vizing}, the graph $G[X]$ has a proper
edge-coloring with $k\leq x_1$ colors.  Every color class is a
matching.

Let $\ell$ be the number of edges in the
discarded classes.
If $k\le y_1$, retain all color classes and put $\ell=0$.
Suppose that $k>y_1$. Discard the $k-y_1$ color classes
of smallest cardinalities. By averaging,
\[
   \ell\le \frac{k-y_1}{k}a_X
        \le \frac{x_1-y_1}{x_1}a_X.
\]
Thus, in either case,
\[
   \ell\le \frac{|x_1-y_1|}{x_1}a_X.
\]  
Lemma~\ref{lem:transfer} and $D\leq n^2/10^4$ give
\[
 |x_1-y_1|\leq2\sqrt{6\delta}\leq2\sqrt{6D},
 \qquad
 x_1\geq n-\sqrt{6D}.
\]
Consequently
\begin{equation}\label{eq:colour-loss}
 \ell\leq
 \frac{2\sqrt{6D}}{n-\sqrt{6D}}a_X
 \leq
 \frac{2\sqrt6/100}{1-\sqrt6/100}a_X
 <\frac{a_X}{18}.
\end{equation}

Assign the retained color classes
injectively to vertices of $Y$.  For an internal edge $vw$ of retained
color $c$, let $z$ be the vertex assigned to $c$, and
call $vwzv$ a candidate triangle.  The candidate is valid
if both $vz$ and $wz$ belong to $E(G)$.

Indeed, their
internal edges are distinct; moreover, they are edge-disjoint.
Moreover, a crossing edge $vz$ can occur only in the
unique retained color assigned to $z$, and at most one edge in that
color class is incident with $v$, because each color class is a
matching.

The number of invalid candidates is at most $\mu$.  Indeed, choose one
missing crossing edge from each invalid candidate.  A fixed missing
edge $vz$ lies in at most one candidate, by the same color-assignment
and matching argument.  The
resulting map is therefore injective. If $T$ denotes the number of valid
candidate triangles, then there
are at least
\begin{equation}\label{eq:number-triangles}
  T\geq a_X-\ell-\mu
\end{equation}
pairwise edge-disjoint valid candidate triangles.

Form a partition from
the $T$ triangles, the remaining two-vertex
cliques, and the partition
of $G[Y]$ supplied by Lemma~\ref{lem:transfer}.  Since
$e_G(X,Y)\leq x_1y_1\leq n^2,$
this gives
\begin{align*}
 p
 &\leq
 T+(a_X-T)+\bigl(e_G(X,Y)-2T\bigr)+5\delta\leq n^2-a_X+2\ell+2\mu+5\delta,
\end{align*}
where the second inequality follows from
\eqref{eq:number-triangles}.  Substituting
$p=n^2-\delta$, $\mu\leq6\delta$, and
\eqref{eq:colour-loss}, we obtain
\[
  a_X\leq2\ell+18\delta
  \leq\frac{a_X}{9}+18\delta.
\]
Therefore,
$a_X\leq\frac{81}{4}\delta\leq36\delta,$
as required.
\end{proof}

Lemmas \ref{lem:forcing} and \ref{lem:repacking} force
$D=\Omega(n^{4/3})$. 

\begin{proof}[Proof of Proposition~\ref{prop:upper-even}]
Recall that
$D=n^2-\bigl(\cpn(G)-\cc(G)\bigr).$
If $D>n^2/10^4$, then
$D>10^{-4}n^{4/3}.$
We may therefore assume that $D\leq n^2/10^4$.

By Lemma~\ref{lem:repacking},
\[
  L=e(G[A])+e(G[B])
  \leq2a_X\leq72\delta\leq72D.
\]
\noindent Lemma~\ref{lem:forcing} yields
$\frac{n^2}{2\sqrt{7D}}-n\leq72D,$
and hence
\begin{equation}\label{eq:master}
  \frac{n^2}{2\sqrt7}
  \leq72D^{3/2}+n\sqrt D.
\end{equation}

Suppose first that $D\geq n/72$.  Since $n\leq72D$, the preceding
lower bound for $L$ gives
$\frac{n^2}{2\sqrt{7D}}\leq144D.$
It follows that
\begin{equation}\label{eq:D-lower}
  D\geq(288\sqrt7)^{-2/3}n^{4/3}.
\end{equation}

Suppose next that $D<n/72$.  The right-hand side of
\eqref{eq:master} is smaller than
$\frac{2}{\sqrt{72}}n^{3/2},$
whereas its left-hand side is $n^2/(2\sqrt7)$.  This is impossible
for all sufficiently large $n$.

Thus, \eqref{eq:D-lower} holds whenever $D\leq n^2/10^4$ and $n$
is sufficiently large.  Taking
\[
  c_0=\min\left\{10^{-4},(288\sqrt7)^{-2/3}\right\}
\]
completes the proof.
\end{proof}

The case for graphs with odd order follows by adding one isolated vertex.

\begin{corollary}\label{cor:upper}
There is an absolute constant $C_2>0$ such that every graph $G$ on
$n$ vertices, with $n$ sufficiently large, satisfies
\[
  \cpn(G)-\cc(G)
  \leq\left\lfloor\frac{n^2}{4}\right\rfloor-C_2n^{4/3}.
\]
\end{corollary}

\begin{proof}
If $n$ is even, the assertion follows directly from
Proposition~\ref{prop:upper-even}.  Suppose that $n$ is odd, and add
one isolated vertex to $G$.  Neither clique parameter changes, so
Proposition~\ref{prop:upper-even}, applied to the resulting
$(n+1)$-vertex graph, gives
\[
 \cpn(G)-\cc(G)
 \leq\frac{(n+1)^2}{4}
 -c_0\left(\frac{n+1}{2}\right)^{4/3}.
\]
Since
$\frac{(n+1)^2}{4}
  -\left\lfloor\frac{n^2}{4}\right\rfloor
  =\frac{n+1}{2},$
the linear term can be absorbed by half of the $n^{4/3}$ term for
all sufficiently large $n$. For $n = 2m$, $\cpn(G)-\cc(G)
 \leq\frac{n^2}{4}
 -c_0\left(\frac{n}{2}\right)^{4/3}.$  For odd $n$, one may take $C_2=\frac{c_0}{2^{7/3}}$.
\end{proof}

\begin{proof}[Proof of Theorem~\ref{thm:main}]
Corollary~\ref{cor:lower}~and
Corollary~\ref{cor:upper} give the lower bound and upper bound, respectively.
\end{proof}

\section*{Declaration on the Use of Generative AI}
The author used ChatGPT (OpenAI) for language polishing, structural organization, references finding, and improving the exposition of some proofs. All mathematical arguments and proofs were independently checked and finalized by the author, who takes full responsibility for the manuscript.


\begin{thebibliography}{99}

\bibitem{Caro1979}
Y.~Caro,
\emph{New results on the independence number},
Technical Report, Tel Aviv University, 1979.

\bibitem{CEOP1985}
L.~Caccetta, P.~Erd\H{o}s, E.~T.~Ordman, and N.~J.~Pullman,
\emph{The difference between the clique numbers of a graph},
Ars Combin.\ \textbf{19A} (1985), 97--106.

\bibitem{Chung1981}
F.~R.~K.~Chung,
\emph{On the decomposition of graphs},
SIAM J. Algebraic Discrete Methods \textbf{2} (1981), 1--12.

\bibitem{Chung1997}
F.~R.~K.~Chung,
\emph{Open problems of Paul Erd\H{o}s in graph theory},
J. Graph Theory \textbf{25} (1997), 3--36.

\bibitem{deBruijnErdos1948}
N.~G.~de Bruijn and P.~Erd\H{o}s,
\emph{On a combinatorial problem},
Indag. Math.\ \textbf{10} (1948), 421--423.

\bibitem{EGP1966}
P.~Erd\H{o}s, A.~W.~Goodman, and L.~P\'osa,
\emph{The representation of a graph by set intersections},
Canad. J. Math.\ \textbf{18} (1966), 106--112.

\bibitem{EFO1988}
P.~Erd\H{o}s, R.~J.~Faudree, and E.~T.~Ordman,
\emph{Clique partitions and clique coverings},
Discrete Math.\ \textbf{72} (1988), 93--101.

\bibitem{EOZ1987}
P.~Erd\H{o}s, E.~T.~Ordman, and Y.~Zalcstein,
\emph{Bounds on threshold dimension and disjoint threshold coverings},
SIAM J. Algebraic Discrete Methods \textbf{8} (1987), 151--154.

\bibitem{EOZ1993}
P.~Erd\H{o}s, E.~T.~Ordman, and Y.~Zalcstein,
\emph{Clique partitions of chordal graphs},
Combin. Probab. Comput.\ \textbf{2} (1993), 409--415.

\bibitem{ErdosProblems632}
\emph{Erd\H{o}s Problems on Graphs},
\emph{Difference of clique partition number to clique covering number},
Department of Mathematics, University of California San Diego,
\href{https://mathweb.ucsd.edu/~erdosproblems/erdos/newproblems/CpMinusCC.html}
{https://mathweb.ucsd.edu/\string~erdosproblems/erdos/newproblems/CpMinusCC.html}
(accessed July 30, 2026).

\bibitem{FRS1997}
R.~J.~Faudree, C.~C.~Rousseau, and R.~H.~Schelp,
\emph{Problems in graph theory from Memphis},
in R.~L.~Graham and J.~Ne\v{s}et\v{r}il (eds.),
\emph{The Mathematics of Paul Erd\H{o}s II},
Algorithms and Combinatorics, 14, Springer, Berlin, 1997, pp.~7--26.

\bibitem{GK1979}
E.~Gy\H{o}ri and A.~V.~Kostochka,
\emph{On a problem of G.~O.~H. Katona and T.~Tarj\'an},
Acta Math. Acad. Sci. Hungar.\ \textbf{34} (1979), 321--327.

\bibitem{Kahn1981}
J.~Kahn,
\emph{Proof of a conjecture of Katona and Tarj\'an},
Period. Math. Hungar.\ \textbf{12} (1981), 81--82.

\bibitem{Mantel1907}
W.~Mantel,
\emph{Problem 28},
Wiskundige Opgaven \textbf{10} (1907), 60--61.

\bibitem{Orlin1977}
J.~B.~Orlin,
\emph{Contentment in graph theory: covering graphs with cliques},
Indag. Math.\ \textbf{39} (1977), 406--424.

\bibitem{Schwartz2022}
S.~Schwartz,
\emph{An overview of graph covering and partitioning},
Discrete Math.\ \textbf{345} (2022), no.~8, Paper No.~112884, 17 pages.

\bibitem{Stinson2004}
D.~R.~Stinson,
\emph{Combinatorial Designs: Constructions and Analysis},
Springer, New York, 2004.

\bibitem{Vizing1964}
V.~G.~Vizing,
\emph{On an estimate of the chromatic class of a $p$-graph},
Diskret. Analiz \textbf{3} (1964), 25--30 (in Russian).

\bibitem{Wallis1982}
W.~D.~Wallis,
\emph{Asymptotic values of clique partition numbers},
Combinatorica \textbf{2} (1982), 99--101.

\bibitem{Wei1981}
V.~K.~Wei,
\emph{A lower bound on the stability number of a simple graph},
Bell Laboratories Technical Memorandum 81-11217-9, 1981.

\end{thebibliography}
\end{document}